\documentclass[11pt,reqno]{amsart}
\usepackage{amsthm,amssymb,amsmath}
\usepackage{xcolor}

\newtheorem{theorem}{Theorem}
\newtheorem*{theorem*}{Theorem}

\newtheorem*{corollary*}{Corollary}
\newtheorem{proposition}{Proposition}

\newtheorem*{claim*}{\sc Claim}

\theoremstyle{definition}

\newtheorem{definition}{\sc Definition}
\newtheorem*{definition*}{\sc Definition}

\newtheorem*{subclasses}{\sc Subclasses}
\newtheorem{remark}{\bf Remark}
\newtheorem*{remark*}{\bf Remark}

\newtheorem*{example*}{\bf Example}

\newcommand{\loc}{{\rm loc}}
\newcommand{\supp}{{\rm sprt\,}}

\newcommand{\Real}{{\rm Re}}

\newcommand{\ord}{{\rm ord}}

\newcommand{\Ord}{{\rm Ord\,}}

\makeatletter
\newcommand{\dotminus}{\mathbin{\text{\@dotminus}}}

\newcommand{\@dotminus}{%
	\ooalign{\hidewidth\raise1ex\hbox{.}\hidewidth\cr$\m@th-$\cr}%
}
\makeatother

\begin{document}
	
	\title[Vanishing of Green's function]{On the vanishing of Green's function, desingularization and Carleman's method}
	
	\author{Ryan Gibara} 
	
	\address{Universit\'{e} Laval, D\'{e}partement de math\'{e}matiques et de statistique, 1045 av.\,de la M\'{e}decine, Qu\'{e}bec, QC, G1V 0A6, Canada}

\curraddr{\sc Department of Mathematical Sciences, P.O. Box 210025, University of
	Cincinnati, Cincinnati, OH 45221--0025, U.S.A.}

\email{ryan.gibara@gmail.com}

	\author{Damir Kinzebulatov}

\address{Universit\'{e} Laval, D\'{e}partement de math\'{e}matiques et de statistique, 1045 av.\,de la M\'{e}decine, Qu\'{e}bec, QC, G1V 0A6, Canada}

\email{damir.kinzebulatov@mat.ulaval.ca}

	\begin{abstract}

	The subject of the present paper is the phenomenon of vanishing of the Green function of the operator $-\Delta + V$ on $\mathbb R^3$ at the points where a potential $V$ has positive critical singularities. More precisely, imposing minimal assumptions on $V$ (i.e.\,the form-boundedness), 
	we obtain an upper bound on the order of vanishing of the Green function. As a by-product of our proof, we improve the existing results on the strong unique continuation for eigenfunctions of $-\Delta + V$ in dimension $d=3$.
	\end{abstract}


\thanks{The research of D.K.\,is supported by the Natural Sciences and Engineering Research Council 
of Canada (grant RGPIN-2017-05567)}

	\maketitle

		%\tableofcontents
	
	\section{Introduction}

	Our motivation for this work goes back to a result of Milman-Sem\"{e}nov \cite{MS0}-\cite{MS2} 
	on a sharp two-sided bound on the heat kernel of the Schr\"{o}dinger operator with positive inverse-square potential in $\mathbb R^d$, $d \geq 3$,
	\begin{equation}
		\label{H}
		H=-\Delta + \delta\frac{(d-2)^2}{4}|x|^{-2}, \quad \delta>0.
	\end{equation}
	They noticed that the semigroup $e^{-tH}$ remains ultracontractive even when considered in the space $L^1(\mathbb R^d,\varphi dx)$ with the vanishing weight $\varphi(x)=|x|^{\beta}\wedge 1$, $\beta=\frac{d-2}{2}(\sqrt{1+\delta}-1)$. In fact, the generator $\varphi H \varphi^{-1}$ of the weighted semigroup $\varphi e^{-tH}\varphi^{-1}$ becomes ``desingularized'' in the context of Nash's method. This observation allowed them to establish a non-Gaussian two-sided bound on the heat kernel of $H$ and, hence the following two-sided bound on the  Green function: 
	\begin{equation}
		\label{green}
		(\mu+H)^{-1}(x,y) \simeq e^{-c\sqrt{\mu}|x-y|}|x-y|^{-d+2} \biggl[1 \wedge \frac{|x||y|}{|x-y|^2} \biggr]^{\beta}.
	\end{equation}
It shows that the singularity of the potential is so strong that it makes the Green function  $x \mapsto (\mu+H)^{-1}(x,y)$ vanish to order $\beta$ at $x=0$ (say, $y \in B^c(0,1)$). 

Generally speaking, the vanishing of the Green function at a point manifests the presence of a critical singularity of the potential.
The relationship between the order of vanishing and the magnitude of the singularity is the subject of this work.

The desingularization procedure allows for the exact calculation of the order of vanishing of the Green function, but it depends on the explicit form of the potential: $V(x)=c\Delta |x|^\beta/|x|^\beta$, where $|x|^\beta$ is the Lyapunov function of $H$ (that is, $H|x|^\beta=0$). 
In the present paper we venture to the other endpoint of the range of possible results: we consider an arbitrary potential $V \in L^1_{\loc}$ in a wide class of locally unbounded potentials for which the self-adjoint operator $$H=-\Delta + V$$  can be defined (i.e.\,the form-bounded potentials, see Definition \ref{def_V}, the sum is in the sense of quadratic forms), and obtain an upper bound on the order of vanishing of the Green function. In such generality, the desingularization method, even if carried out for some Lyapunov function $\psi$, i.e.\,$H\psi=0$, would give little information: one ends up with an essentially equally difficult problem of estimating the order of vanishing of $\psi$. Thus, another approach is needed.

It will be convenient to estimate the order of vanishing of $u=(\mu+H)^{-1}f$, for $f$ identically zero in $B(0,1)$, rather than that of $x \mapsto (\mu+H)^{-1}(x,y)$, $y \in B^c(0,1)$. 
The main result of this paper, stated briefly, is as follows. 

\begin{theorem*} Let $d=3$.
Let $V$ be a form-bounded potential with a sufficiently small form-bound, $H=-\Delta + V$, and let $u:=(\mu+H)^{-1}f$ for $f$ identically zero in $B(0,1)$. If $u$ vanishes in $L^p$ at $x=0$ to order at least $\beta>0$, see definition below (with $\beta$, in some sense, substantial), then
			\begin{equation}
			\label{bd__}
			\||x|^{-[\beta]-1} u \|_{L^p(B(0,1))} \leq K
		\end{equation}
		where $[\beta]$ is the integer part of $\beta$.
\end{theorem*}

For the detailed statement, see Theorem \ref{thm1} below. One easily obtains from \eqref{bd__} e.g.

\begin{corollary*}
	$
	\ord^p_{x=0}u \leq \log_{1/a}\frac{K}{a\|u\|_{L^p(B(0,a))}}$, $0<a<1$.
	\end{corollary*}

	(See Remark 2 in Section \ref{discussion_sect}.)		So, if $\|u\|_{L^p(B(0,a))}$ is not too small, then the order of vanishing of $u$ cannot be too large. On the other hand, if the order of vanishing is large, then $\|u\|_{L^p(B(0,a))}$ must be small. Results of this type have appeared in the literature, see \cite{DZ, KT, MVs} and references therein, although there the authors treat considerably less singular potentials.

It should be noted that the effect of vanishing and ensuing regularity of weak solutions $u$ to $(-\Delta + V)u=0$, as well as of weak solutions to the corresponding parabolic equation, with supercritical positive potential $V(x)=c\frac{1}{|x|^{2+\gamma}}$, $\gamma >0$ (obviously, not form-bounded) was studied recently by Li-Zhang \cite{LZ}.

	\subsection{Comments}
	1. We prove estimate \eqref{bd__} using Carleman's method: one recognizes \eqref{bd__} as a finite unique continuation-type statement. In fact, the constant $K$ in \eqref{bd__} does not depend on $[\beta]$,
so if $u$ vanishes to infinite order at $x=0$  then taking $\beta \rightarrow \infty$ we obtain $u \equiv 0 $ on $B(0,1)$. Thus, we obtain as a by-product a strong unique continuation (SUC) result:
\begin{center}
``$u$ vanishes to infinite order at a point $\Rightarrow$ $u \equiv 0$ everywhere.''
\end{center}
The corresponding SUC result for solutions of the differential inequality $|\Delta u| \leq |Vu|$ with form-bounded $V$ was obtained \cite{KSh}. So, to prove the theorem, one is tempted to take the corresponding estimate of type \eqref{bd__} from \cite{KSh} and call it a day. Unfortunately, this leads to an unsatisfactory result. The reason is that  the proof of the SUC in \cite{KSh} requires an additional to the form-boundedness assumption $V \in L^{1+\varepsilon}_{\loc}$, $\varepsilon>0$, and the resulting upper bound on the order of vanishing from \cite{KSh} is unnatural: it tends to infinity as $\varepsilon \downarrow 0$ (note that one should be able to take $\varepsilon=0$ since for any $\varepsilon>0$ there are form-bounded $V \not \in L^{1+\varepsilon}_{\loc}$). The problem is not technical but ideological. Namely, the proof of SUC in \cite{KSh} uses in an essential manner interpolation Sobolev-type inequalities to control the error terms. As \cite{KSh} themselves demonstrate in their proof of the weak unique continuation for $|\Delta u| \leq |Vu|$ (``vanishing in an open set $\Rightarrow$ vanishing everywhere'') avoiding the use of Sobolev-type inequalities and resorting instead to the use of $L^2 \rightarrow L^2$ bounds allows to relax the assumptions on $V$ to mere form-boundedness. In the present paper we exclude the additional assumption $V \not \in L^{1+\varepsilon}_{\loc}$ (and the unnatural dependence of the upper bound on $\varepsilon>0$) by eliminating the use of Sobolev-type inequalities. In a sense, our proof of Theorem \ref{thm1} is closer to the idea of \cite{KSh} than their own proof of the SUC.  The new proof, however, required a rather substantial modification of the argument in \cite{KSh}. In particular, we now employ  estimate \eqref{e2} in addition to \eqref{e1} to control the ``error term'' of gradient type $\nabla(-\Delta)^{-1}E_j(u)$ (it is when estimating this term that we use the hypothesis on the order of vanishing of $u$).

\smallskip

2. The novelty in this paper concerns dimension $d=3$. In dimension $d \geq 4$, similarly to \cite{KSh}, we would have to impose a stronger requirement that $|V|^{(d-1)/2}$ is form-bounded with respect to $(-\Delta)^{(d-1)/2}$,
 so we discuss the corresponding result only briefly, see Remark \ref{rem_d} below.

\smallskip
	
	3.	Both the singular weight $|x|^{-[\beta]-1}$ in the Carleman method and the Lyapunov function $|x|^\beta$ in the desingularization method are born out of the Riesz potential $(-\Delta)^{-1}(x,y)$. 
Both approaches to estimating the order of vanishing have a common core that we touched in this work, but that still needs to be fully understood and exploited. 

\smallskip
	
4.~In \cite{KSh} the authors stay in $L^2$, while for the purpose of estimating the order of vanishing of the Green function it is more natural to work in $L^p$, cf.\,Remark \ref{rem_p} below. However, we keep intact the $L^2 \rightarrow L^2$ assumption	on $V$ (i.e.\,the form-boundedness), and thus its class of singularities. 

\smallskip

5. Our proof of \eqref{bd__} applies to $|\Delta u| \leq |Vu|$ and yields a SUC result  in dimension $d=3$  that contains, to the best of our knowledge, all the existing results on the SUC in the spaces
of solutions $u$ large enough to contain the eigenfunctions of $-\Delta + V$, including the classical results of
Sawyer \cite{S} ($V$ is in the Kato class), Jerison-Kenig \cite{JK} ($V$ in $L^{d/2}_{\loc}$) and Stein \cite{St} ($V$ in the weak-$L^{d/2}$ space). (Historically, the principal motivation behind the efforts to prove UC for $|\Delta u| \leq |Vu|$ with singular $V$ is the problem of absence of positive eigenvalues of $-\Delta + V$ in $\mathbb R^d$.) We state the corresponding result in Section \ref{discussion_sect}.

\smallskip

6.~When estimating the order of vanishing of the Green function, in principle we face a problem that is fundamentally simpler than the problem of proving unique continuation for eigenfunctions of the Schr\"{o}dinger operator $H$. Indeed, unlike the Green function, the eigenfunctions can be widely oscillating; the key feature of Carleman's method is that it allows to combat these oscillations via singular weights. 
On the other hand, when proving unique continuation, one is working with a function that is \textit{actually} identically zero, while the Green function is non-trivial and  can vanish to a finite order, cf.\,\eqref{H}.

See further discussion in Section \ref{discussion_sect}.

\subsection{About the proof}
Put $u=\eta_j u$, where $\eta_j$ is an appropriate cutoff function identically equal to $0$ in $B(0,j^{-1})$ ($j$ will be taken to $\infty$ to take into account that $u$ vanishes at $x=0$). We have
	\begin{equation}
	\label{id44}
		u_j=(-\Delta)^{-1}(-\Delta u_j)
	\end{equation}
which yields
	\begin{equation}
	\label{id45}
		u_j =(-\Delta)^{-1}\eta_j(-\Delta u) + \nabla(-\Delta)^{-1}E_j(u) + (-\Delta)^{-1}\tilde{E}_j(u)
	\end{equation}
	for appropriate ``error terms'' $E_j$, $\tilde{E}_j$ that depend on $u$ but not on its derivatives (this is important since we want to avoid applying Sobolev-type inequalities to control vanishing of $\nabla u$). Set $N:=[\beta]+1$.
Now we subtract the $(N-1)$-degree Taylor polynomial at $x=0$ from both sides of identity \eqref{id45}. Since $u_j$ is identically zero around the origin, this will not change the left-hand side of \eqref{id45}, but it will introduce 
a singular weight $\varphi_N(x):=|x|^{-N}$ into the right-hand side via estimate \cite{S}
$$\bigl|(-\Delta)^{-1}(x,y) - T^{N-1}_{x=0}\big((-\Delta)^{-1}(x,y)\big) \bigr| \leq 
C\frac{\varphi_N(y)}{\varphi_N(x)}(-\Delta)^{-1}(x,y),
$$
(and similarly for $\nabla_{x_i}(-\Delta)^{-1}(x,y)$, see \eqref{e2} below).
Thus one obtains from \eqref{id45}
\begin{equation}
\label{i7}
|u_j| \leq C \varphi_N^{-1} (-\Delta)^{-1} \varphi_N \eta_j|\Delta u| + \text{ error terms}.
\end{equation}
Multiplying \eqref{i7} by $\mathbf{1}_{B(0,1)} |V|^{\frac{1}{2}}\varphi_N$ (say, working in $L^2$), we arrive at
\begin{align*}
\|\mathbf{1}_{B(0,1)} |V|^{\frac{1}{2}}\varphi_N u_j\|_2 & \leq C\|\mathbf{1}_{B(0,1)} |V|^{\frac{1}{2}}(-\Delta)^{-1}\eta_j \varphi_N |\Delta u|\|_2 + \text{ error terms 2} \\
& (\text{$-\Delta u = -Vu+f-\mu u$ }) \\
& \leq C\|\mathbf{1}_{B(0,1)} |V|^{\frac{1}{2}}(-\Delta)^{-1}\varphi_N |V||u_j|\|_2 + \text{ error terms 3}
\end{align*}
hence
\begin{align}
 \|\mathbf{1}_{B(0,1)} |V|^{\frac{1}{2}}\varphi_N u_j\|_2 \leq C\|\mathbf{1}_{B(0,1)} |V|^{\frac{1}{2}}(-\Delta)^{-1}|V|^{\frac{1}{2}}\mathbf{1}_{B(0,1)}\|_{2 \rightarrow 2}\,& \|\mathbf{1}_{B(0,1)} |V|^{\frac{1}{2}}\varphi_N u_j\|_2 \label{i8} \\
& + \text{ error terms 4}. \notag
\end{align}
The local form-boundedness condition on $V$ is $\|\mathbf{1}_{B(0,1)} |V|^{1/2}(-\Delta)^{-1}|V|^{\frac{1}{2}}\mathbf{1}_{B(0,1)}\|_{2 \rightarrow 2} \leq \nu$, so \eqref{i8} yields
$$
\|\mathbf{1}_{B(0,1)} |V|^{\frac{1}{2}}\varphi_N u_j\|_2 \leq C\nu \|\mathbf{1}_{B(0,1)} |V|^{\frac{1}{2}}\varphi_N u_j\|_2 + \text{ error terms 5},
$$
i.e.
$$
(1-C\nu)\|\mathbf{1}_{B(0,1)} |V|^{\frac{1}{2}}\varphi_N u_j\|_2 \leq \text{ error terms 5}.
$$
Without loss of generality $|V| \geq 1$, so, provided that the form-bound $\nu$ is so small that $1-C\nu>0$, the last inequality gives 
\eqref{bd__} upon estimating properly the error terms.

To the best of our knowledge, the idea of subtracting Taylor polynomial from both sides of \eqref{id44} to prove unique continuation first appeared in Sawyer \cite{S}.

\subsection{Notations}
We denote the $L^p$-norm by $\|f\|_p$ and the norm of an operator $T:L^p\rightarrow L^q$ by $\|T\|_{p \rightarrow q}$. 
The open ball with centre $x_0\in\mathbb R^d$ and radius $r>0$ is denoted by $B(x_0,r)$, 
its complement by $B^c(x_0,r)$, and its closed counterpart by $\bar{B}(x_0,r)$. The indicator function of the ball is written $\mathbf{1}_{B(x_0,r)}$ and, in the particular case when $x_0=0$, we employ the shorthand $\mathbf{1}_{r}:=\mathbf{1}_{B(0,r)}$.
Put
$$
\langle f \rangle:=\int_{\mathbb R^d} f dx, \quad \langle f,g\rangle:=\langle f\bar{g}\rangle
$$

\bigskip

\section{Main result}

	\begin{definition}
	\label{def_V}
		A potential $V \in L^1_{\loc}$ is said to be form-bounded if there exists $\delta>0$ such that the following quadratic form inequality holds:
		\begin{equation*}
			\langle |V|\varphi,\varphi\rangle \leq \delta \langle \nabla \varphi,\nabla \varphi\rangle + c_\delta\langle \varphi,\varphi\rangle
		\end{equation*}
		for all $\varphi \in C_c^\infty$, for some constant $c_\delta \geq 0$ (written as $V \in \mathbf{F}_\delta$).

	\end{definition}

	The constant $\delta$ is called the form-bound of $V$. 
Equivalently, $V \in \mathbf{F}_\delta$  can be re-stated as $$
		\||V|^{\frac{1}{2}}(\lambda-\Delta)^{-\frac{1}{2}}\|_{2 \rightarrow 2} \leq \sqrt{\delta},$$
		where $\lambda=\frac{c_\delta}{\delta}$.
		
The assumption $V \in \mathbf{F}_\delta$ with $\delta<1$ ensures that the symmetric form $t[u,v]=\langle \nabla u,\nabla v\rangle + \langle Vu,v\rangle$, $D(t)=W^{1,2}$ is semi-bounded from below and closed, and hence determines a unique self-adjoint operator $H$,
$$t[u,v]=\langle Hu,v\rangle, \quad v \in D(t), \quad u \in D(H) \subset W^{1,2} \cap \{u \in L^2 \mid |V|^{\frac{1}{2}}u \in L^2\},$$ 
denoted by
$$H=-\Delta \dotplus V \quad (\text{the form-sum of $-\Delta$ and $V$}),$$
see e.g.\,\cite[Ch.VI]{Ka}. 
(When constructing a self-adjoint realization of $-\Delta + V$ in $L^2$ one should distinguish between the positive and the negative parts of $V$ or, moreover, take into account the cancellation phenomena \cite{MVe}. However, we are interested here in potentials whose positive part is larger than the negative part so that the Green function vanishes rather than blows up, so the fact that we impose a constraint on $|V|$, as is dictated by the method, does not appear to be too restrictive.)

\begin{subclasses}
The following are some sub-classes of $\mathbf{F}_\delta$ defined in elementary terms (listed in the increasing order): 

1) $L^\frac{d}{2}$ class (the inclusion follows easily from the Sobolev inequality);

2)  weak $L^{\frac{d}{2}}$ class (see \cite{KPS} for the proof of inclusion $\subset \mathbf{F}_\delta$), e.g.\,$V(x)=\delta\frac{(d-2)^2}{4}|x|^{-2} \in \mathbf{F}_\delta$ with $c_\delta=0$;

3) Campanato-Morrey class ($s>1$),
$$
\left\{V \in L_{\loc}^s: \biggl(\frac{1}{|Q|}\int_Q |V(x)|^s dx \biggr)^{\frac{1}{s}} \leq c_s l(Q)^{-2} \text{ for all cubes $Q$}\right\},
$$
$|Q|$ and $l(Q)$ are the volume and the side length of a cube $Q$, respectively;

4) Chang-Wilson-Wolff class ($s>1$),
$$
\left\{V \in L_{\loc}^s: \sup_Q \frac{1}{|Q|}\int_Q |V(x)|\, l(Q)^2 \varphi\big(|v(x)|\,l(Q)^2 \big) dx<\infty\right\},
$$
where 
$\varphi:[0,\infty[ \rightarrow [1,\infty[$ is an increasing function such that
$
\int_1^\infty \frac{dx}{x\varphi(x)}<\infty.
$
See \cite{CWW} for the proof of inclusion of this class into $\mathbf{F}_\delta$.

In 1) the form-bound $\delta$ can be chosen arbitrarily small, while in 2)-4) $\delta$ depends on the norm of $V$ in these classes.
\end{subclasses}

Throughout the paper:
\begin{equation}
\label{hyp}
\tag{$C$}
\begin{array}{c}
V \in \mathbf{F}_\delta \text{ with } \delta<1, \\[3mm]
\text{$u$ is the solution to $\mu u + Hu=f$ for  $f \in L^1 \cap L^\infty$, $f=0$ in $B(0,1)$, $\mu>\frac{c_\delta}{\delta}$}.
\end{array}
\end{equation}

In order to work in $L^p$	while keeping intact the form-boundedness (i.e.\,$L^2 \rightarrow L^2$) assumption	on $V$, we will need the following result of Beliy-Sem\"{e}nov \cite{BS}: the form-sum $H=-\Delta \dotplus V$ admits a realization in $L^p$, $p \in ]p_-,p_+[$, $p_{\pm}:=\frac{2}{1 \mp \sqrt{1-\delta}}$ as the generator $H_p$ of the $C_0$ semigroup 
	$$
	e^{-tH_p}:=\bigl[e^{-tH} \bigr]_{L^p \rightarrow L^p}^{\rm clos} \quad (\text{the closure of operator}).
	$$
The interval $]p_-,p_+[$ is sharp. In particular, $u=(\mu+H)^{-1}f$ is in general not in $L^\infty_{\loc}$, even if $f \in C_c^\infty$.

	\begin{definition}
		\label{def_o}

		A function $u \in L^p_{\loc}$ is said to vanish in $L^p$ at $x \in \mathbb R^d$  to order $\beta> 0$ if 
		\begin{equation}
			\label{v1}
			\lim_{r \downarrow 0}\frac{1}{r^s}\big\langle \mathbf{1}_{B(x,r)}|u|^p\big\rangle=0 \quad \text{for every $0 < s < d+p\beta$}.
		\end{equation}
	\end{definition}
	
			The supremum of such $\beta$, called the order of vanishing of $u$ in $L^p$ at $x$, will be denoted by $\ord^p_{x}u$. 				If $p=2,$ we write simply $\ord_x u$.

	For example, if $u(x)=|x|^\alpha$, $\alpha>0$, then $\ord^p_{x=0}|x|^\alpha=\alpha$.

\medskip

	\begin{theorem}
		\label{thm1}
		Let $d=3$.
Assume \eqref{hyp}. If, additionally,
				$$
		\|\mathbf{1}_{B(0,3)}|V|^\frac12(-\Delta)^{-\frac{1}{2}}\|_{2 \rightarrow 2} \leq \sqrt{\nu}
		$$		
		with a sufficiently small local form-bound $\nu$, then the following is true.
		
		If $u$ vanishes in $L^p$ at $x=0$ 
		for some $p \in [2,p_+[$ and its order of vanishing $\ord^p_{x=0}u$ is ``substantial'' in the sense that we can fix a positive $\beta \not\in \mathbb Z$, $\beta \leq \ord^p_{x=0}u$, with the property that $p([\beta]+1-\beta)<1,$
		then
		\begin{equation}
			\label{est_m}
			\tag{$\star$}
			\|\mathbf{1}_{B(0,1)} |x|^{-[\beta]-1}  u \|_p \leq K,
		\end{equation}
		where the constant $K=K(\|f\|_p,\|f\|_2,\nu)<\infty$ is independent of $\beta$.

	\end{theorem}

A few comments are in order:

1. If $V \in \mathbf{F}_\delta$, $\delta<1$, the resolvents $(\mu+H)^{-1}$, $\mu>\frac{c_\delta}{\delta}$ are integral operators. This result is due to Sem\"{e}nov \cite{Se} who proved it by verifying Bukhvalov's criterion (in the situation where the Dunford-Pettis Theorem is inapplicable, since $e^{-tH}$ is not $L^2 \rightarrow L^p$ bounded for $p=\infty$). 

2.	If $\ord^p_{x=0}u>n+\frac{1}{p}$ for an integer $n \geq 0$, then  we can fix $\beta>n+\frac{1}{p}$ arbitrarily close to $n+\frac{1}{p}$.

3.	If $u$ vanishes to infinite order, then $u=0$ in $B(0,1)$ by \eqref{est_m} (in this regard, see Section \ref{discussion_sect}). 

4. In fact, we prove a stronger result:
$$
			\|\mathbf{1}_{B(0,1)} (|V| + 1)^{\frac{1}{p}} |x|^{-[\beta]-1}  u \|_p \leq K.
			$$

	\begin{remark}[Regarding $d \geq 4$]
\label{rem_d}
In the case $d \geq 4$ we have to impose, as in \cite{KSh}, a more restrictive assumption on $V$ (that, nevertheless, includes e.g.\,$V \in L^{\frac{d}{2},\infty}$):
				\begin{equation}
				\label{req_d}
	V \in L^{\frac{d-1}{2}}(\bar{B}(0,3)) \quad \text{ and } \quad	\|\mathbf{1}_{B(0,3)}|V|^{\frac{d-1}{4}}(-\Delta)^{-\frac{d-1}{4}}\|_{2 \rightarrow 2} \leq \sqrt{\nu}.
		\end{equation}
The reason is that in dimensions $d \geq 4$ the key bounds \eqref{e1}, \eqref{e2} are valid for $(-\Delta)^{\frac{d-1}{2}}$ rather than $-\Delta$. The assumption \eqref{req_d} then allows to run appropriate interpolation arguments, see \cite{KSh}. On the one hand, due to this more restrictive condition \eqref{req_d} on $V$, the issue with the extra assumption $V \in L^{1+\varepsilon}_{\loc}$ discussed in the introduction is not present. On the other hand, by assuming \eqref{req_d} and following the proof of the SUC in \cite{KSh}, one still gets an upper bound on the order of vanishing that is rather unnatural in low dimensions but improves as $d \rightarrow \infty$:
\textit{If \eqref{hyp} holds with $\delta$ small enough so that $s:=2\frac{d(d-1)}{d^2-d-4}<p_+$ and $V$ satisfies \eqref{req_d}
		with $\nu$ sufficiently small, then the following is true.}
		\textit{If $u$ vanishes in $L^s$ at $x=0$ 
and its order of vanishing $\ord^s_{x=0}u$ is substantial in the sense that we can fix a positive $\beta \not\in \mathbb Z$, $\beta \leq \ord^s_{x=0}u$, with the property that
$$[\beta]-\beta+2+\bigl(\frac{d}{2}-\frac{1}{2}\bigr)\frac{d-3}{d-1} <\frac{d}{s},$$
		then
$
			\|\mathbf{1}_{B(0,1)} |x|^{-[\beta]-1}  u \|_2 \leq K,
$
		where $K=K(\|f\|_p,\|f\|_2,\nu)<\infty$ is independent of $\beta$ (see Remark \ref{proof_d} for the proof).} 
One can improve this bound by reworking the proof of the SUC in \cite{KSh} along the lines of the proof in the present paper, i.e.\,excluding any essential use of Sobolev-type inequalities. We will not do it here to 
keep the paper short, also because we do not have anything to add here in what concerns a more important problem: to weaken \eqref{req_d} to  $\|\mathbf{1}_{B(0,3)}|V|^{\frac{1}{2}}(-\Delta)^{-\frac{1}{2}}\|_{2 \rightarrow 2} \leq \sqrt{\nu}$.
\end{remark}

Let us comment on the existence of a lower bound on the order of vanishing. Consider potential
	\begin{align}
		\label{V_m}
		V(x) =\delta\frac{(d-2)^2}{4}|x|^{-2} + V_0(x) \text{ with } V_0 \in \mathbf{F}_{\delta_0}, 
	\end{align}
	with $\delta_0/\delta$ assumed to be sufficiently small.

	\begin{definition}
		\label{def_O}
		For $u \in L^p_{\loc}$,
		denote by $\Ord^{p}_{x}u$ the supremum of $\beta_1 > 0$ such that 
		\begin{equation*}
			\big\langle \mathbf{1}_{B(x,1)}|\cdot-x|^{-s_1}|u|^p\rangle <\infty \quad \text{for every $0 < s_1< d+p\beta_1$}.
		\end{equation*}

	\end{definition}
	
			Analogously, for every $\alpha>0$, $\Ord^p_{x=0}|x|^\alpha=\alpha$.
	
	\medskip
	
	If one is willing to replace $\ord^p_{x}u$ by $\Ord^p_{x}u$, then the problem of finding a lower bound on the order of vanishing of $u$ for potential $\eqref{V_m}$ becomes easy (see Appendix \ref{appA}):
	\begin{equation}
		\label{bd_0}
		\Ord_{x=0}^{\frac{2d}{d-2}}u \geq \frac{d-2}{2}(\sqrt{1+\delta-\delta_0}-1).
	\end{equation}
	
\begin{remark}	
\label{rem_p}
If $\delta_0=0$, then we get the same lower bound on $\Ord_{x=0}^{\frac{2d}{d-2}}u$ (in fact, the equality) as we would get from \eqref{green}. In this regard, we note that $\Ord_{x=0}^{p}u$ with $p=2$ gives a suboptimal result.
\end{remark}

	On the other hand, it is trivial to see that
	\begin{equation}
		\label{bd_}
		\Ord^p_{x}u \leq \ord^p_{x}u.
	\end{equation}
The question arises: are $\ord_x^p u$ and $\Ord_x^p u$ comparable on solutions to $\mu u + Hu=f$?
Combined with \eqref{bd_}, the estimate \eqref{est_m} of Theorem \ref{thm1} yields in dimension $d=3$: for a given $\beta \not\in \mathbb Z$	
	$$\beta \leq \ord^p_{x=0}u,\;\;p([\beta]+1-\beta)<1 \quad \Rightarrow \quad
	[\beta]+1-\frac{3}{p} \leq \Ord_{x=0}^pu.
	$$

	\medskip

	\bigskip
	
	\section{Key estimates}

	We begin with few general definitions and results, valid in all dimensions $d \geq 3$. 
	Set $$(-\Delta)^{-\frac{\alpha}{2}}(x,y)=c_{d,\alpha}|x-y|^{-d+\alpha} \quad (0<\alpha <d), \quad x,y \in \mathbb R^d, x \neq y,
	$$
	$$
	\text{where}\quad c_{d,\alpha}:=\Gamma\big(\frac{d-\alpha}{2}\big)\big(\pi^{\frac{d}{2}}2^\alpha\Gamma\big(\frac{d}{2}\big) \big)^{-1},$$
	so $(-\Delta)^{-\frac{\alpha}{2}}f(x)=\big\langle (-\Delta)^{-\frac{\alpha}{2}}(x,\cdot)f(\cdot)\big\rangle$ for $f \in C_c$. 
	
	For $N \geq 1$, we define the ``truncated Riesz potential''
	$$
	[(-\Delta)^{-\frac{\alpha}{2}}]_N(x,y):=(-\Delta)^{-\frac{\alpha}{2}}(x,y) - T^{N-1}_{x}\bigl((-\Delta)^{-\frac{\alpha}{2}}(x,y)\bigr),
	$$
	where $T^{N-1}_{x}$ stands for the $(N-1)$-degree Taylor polynomial in the variable $x$ at $x=0$.

	Define in an analogous way $[\nabla_{x_i} (-\Delta)^{-\frac{d-1}{2}}]_N(x,y)$, $1 \leq i \leq d$.
	
	We put
	$$
	[(-\Delta)^{-1}]_Nf(x):=\big\langle  [(-\Delta)^{-1}]_N(x,\cdot)f(\cdot)\big\rangle, \quad f \in C_c,
	$$
	and define similarly operator $[\nabla_{x_i} (-\Delta)^{-\frac{d-1}{2}}]_N$.

	\medskip
	
	1. The following two estimates will play a crucial role. Define the singular weight
	$$\varphi_t(x):=|x|^{-t}, \quad t>0.$$

	\begin{proposition} 
		\label{sawyer}
		There exist constants $C_1=C_1(d)$ and $C_2=C_2(d)$ such that, for every $N \geq 1$, 
		\begin{equation}
			\label{e1}
			\tag{{\rm $S_1$}}
			\big|[(-\Delta)^{-\frac{d-1}{2}}]_N(x,y)\big|\leq C_1 \frac{\varphi_{N}(y)}{\varphi_{N}(x)}(-\Delta)^{-\frac{d-1}{2}}(x,y),
		\end{equation}
		\begin{equation}
			\label{e2}
			\tag{{\rm $S_2$}}
			\big | [\nabla_{x_i} (-\Delta)^{-\frac{d-1}{2}}]_N(x,y) \big| \leq C_2 N  \frac{\varphi_{N}(y)}{\varphi_{N}(x)}(-\Delta)^{-\frac{d-2}{2}}(x,y)
		\end{equation}
		for each $1 \leq i \leq d$, for all $x,y \in \mathbb R^d$, $x \neq y$, $y \neq 0$.
	\end{proposition}

	%A crucial feature of both estimates is, of course, that the constants do not depend on $N$.
	
	The first estimate \eqref{e1} is proved in Sawyer \cite{S}. The proof of the second estimate \eqref{e2} is obtained via a modification of the proof of \eqref{e1} in \cite{S}:
	
	\begin{proof}[Proof of \eqref{e2}]
		We will consider the case $|x| < |y|$. The  case $|x| \geq |y|$ is dealt with in the same way as in \cite{S}.
		
		Following \cite{S}, we use dilation and rotation to reduce our task to the proof of the following two estimates in the complex plane: for all $|z|<1$
		\begin{equation}
			\label{z2}
			\big | \partial_z |1-z|^{-1}-T^N_{z,\bar{z}}(\partial_z |1-z|^{-1})\big| \leq C N |z|^N|1-z|^{-2},
		\end{equation}
		\begin{equation}
			\label{z3}
			\big | \partial_{\bar{z}} |1-z|^{-1}-T^N_{z,\bar{z}}(\partial_{\bar{z}} |1-z|^{-1})\big| \leq C N |z|^N|1-z|^{-2}.
		\end{equation}

		Let us prove e.g.\,\eqref{z2}. Representing $|1-z|^{-1}=(1-z)^{-\frac{1}{2}}(1-\bar{z})^{-\frac{1}{2}}$, we have
		\begin{equation*}
			2\partial_z |1-z|^{-1}=(1-z)^{-1}(1-z)^{-\frac{1}{2}}(1-\bar{z})^{-\frac{1}{2}},
		\end{equation*}
		so we can expand
		$$
		2\partial_z |1-z|^{-1}=(1+z+z^2+\dots)\sum_{n+m \geq 0}^\infty c_{n,m}z^n\bar{z}^m.
		$$
		%(Let us note that the power series $\sum_{n+m \geq 0}^\infty c_{n,m}z^n\bar{z}^m\;(=|1-z|^{-1})$ was analysed in detail in \cite{S}. We will use the corresponding results below.) 
		We obtain $ \partial_z |1-z|^{-1}-T^N_{z,\bar{z}}(\partial_z |1-z|^{-1})$ from the previous expansion by excluding the terms of order $ \leq N-1$:
		\begin{align*}
			&	2\partial_z |1-z|^{-1}  -T^N_{z,\bar{z}}(2\partial_z |1-z|^{-1}) \\
			=& \sum_{n+m \geq N} c_{n,m}z^n\bar{z}^m + z \!\sum_{n+m \geq N-1} c_{n,m}z^n\bar{z}^m  + \dots + z^{N-1} \!\sum_{n+m \geq 1} c_{n,m}z^n\bar{z}^m \\
			& + (z^N+z^{N+1}+\dots) \sum_{n+m \geq 0} c_{n,m}z^n\bar{z}^m.
		\end{align*}
		By \eqref{e1}, $|\sum_{n+m \geq K} c_{n,m}z^n\bar{z}^m| \leq C|z|^K |1-z|^{-1}.$ Therefore,
		\begin{align*}
			& \bigl|2\partial_z |1-z|^{-1}-T^N_{z,\bar{z}}(2\partial_z |1-z|^{-1}) \bigr| \\
			& \leq C|z|^N |1-z|^{-1} + C |z| |z|^{N-1} |1-z|^{-1} + \dots + C |z|^{N-1} |z| |1-z|^{-1} + |z|^N \frac{1}{|1-z|} |1-z|^{-1},
		\end{align*}
		which yields (upon redefining $C$)
		$$
		\bigl|\partial_z |1-z|^{-1}-T^N_{z,\bar{z}}(\partial_z |1-z|^{-1}) \bigr| \leq C N |z|^N|1-z|^{-2},
		$$
		as needed.		
		%The proof of \eqref{z3} proceeds in much the same way.
	\end{proof}

	2. The next proposition is a special case of a result in \cite{BS} (see also \cite{LS}). 
	
	\begin{proposition}
		\label{prop_p}
		Assume that $0 \leq V \in L^1_{\loc}$ and $$\|V^{\frac{1}{2}}(-\Delta)^{-\frac{1}{2}}\|_{2 \rightarrow 2} \leq \sqrt{\nu}$$ 
		for some $\nu>0$. Then, for every $p \in ]1,\infty[$,
		$$\|V^{\frac{1}{p}}(-\Delta)^{-1}V^{\frac{1}{p'}}\|_{p \rightarrow p} \leq \kappa_p\nu, \quad \kappa_p=\frac{pp'}{4},
		$$
		where $p'=\frac{p}{p-1}$.
	\end{proposition}
	
	For the reader's convenience, we include the proof.
	
	\begin{proof}[Proof of Proposition \ref{prop_p}]
		It suffices to carry out the proof for a bounded $V$ having compact support, and then use  Fatou's Lemma.
		Define $A = -\Delta$, $D(A)= W^{2,2}$ and $A_p=-\Delta$, $D(A_p)= W^{2,p}$.
		Since $A$ is a Markov generator, $e^{-tA} \upharpoonright L^2 \cap L^p=e^{-tA_p} \upharpoonright L^2 \cap L^p$, we have by \cite[Theorem 2.1]{LS}
		$$
		0 \leq u \in D(A_p) \Rightarrow v := u^\frac{p}{2} \in D(A) \text{ and } \kappa_p^{-1} \|A^\frac{1}{2} v \|_2^2 \leq \langle A_p u, u^{p-1} \rangle.
		$$
		Now, let $u$ be the solution of $A_p u = V^{\frac{1}{p'}}|f|$, $f \in L^1 \cap L^\infty.$ Then, by our assumption on $V$,
		\[
		(\kappa_p \nu)^{-1} \| V^\frac{1}{2} v \|_2^2 \leq \langle A_p u, u^{p-1} \rangle,
		\]
		i.e.\,$
		(\kappa_p \nu)^{-1} \| V^\frac{1}{p} u\|_p^p \leq \langle A_p u, u^{p-1} \rangle$.
		Hence $\| V^\frac{1}{p} u\|_p^p \leq \kappa_p\nu\|f\|_p \|V^{\frac{1}{p'}}u^{p-1}\|_{p'} = \kappa_p\nu\|f\|_p \|V^{\frac{1}{p}}u\|_p^{p-1}$,
		and the result follows. 
	\end{proof}

	\bigskip
	
	\section{Proof of Theorem \ref{thm1}}

	Throughout this section, $d=3$. We will keep writing $d$ rather than $3$  to make the argument easier to follow. We introduce some notations first.
	
	\smallskip
	
	1.~We fix smooth cutoff functions $0 \leq \eta  \leq 1$ and $0\leq\xi_j\leq 1$ ($j \geq 1$) such that
	$$\eta =1 \text{ in } B(0,2), \quad \eta = 0 \text{ in } B^c(0,3), \quad |\nabla \eta| \leq c, \quad |\Delta \eta| \leq c^2,$$
	$$\xi_j= 1 \text{ in } B^c(0,2j^{-1}), \quad \xi_j= 0  \text{ in } B(0,j^{-1}),  \quad |\nabla\xi_j(x)|\leq cj, \quad |\Delta\xi_j(x)|\leq c^2 j^2$$
	for a generic constant $c$. 
	Set $$\eta_j:=\xi_j \eta, \quad j \geq 2.$$

	\medskip
	
	2.~We have $\mu u  + Hu=f$. 
	By a standard result, $D(H)$ is contained in the domain of the maximal operator $D(H_{\max})=\{v \in L^2 \mid Vv \in L^1_{\loc}, -\Delta v + Vv \in L^2\}$, so 
	\begin{equation}
		\label{delta}
		-\Delta u \in L^1_{\loc} \quad \text{ and } -\Delta u = -Vu- \mu u + f \text{ a.e.}
	\end{equation}
	
	\medskip
	
	3.~We start the proof of \eqref{est_m}. We have
	\begin{equation*}
		%\label{id1}
		u_j=(-\Delta)^{-1}(-\Delta u_j),
	\end{equation*}
	where $u_j:=u \eta_j$. We evaluate
	\begin{equation*}
		-\Delta u_j =\eta_j(- \Delta u) + \nabla  E_j(u) + \tilde{E}_j(u),
	\end{equation*}
	where 
	$$
	E_j(u):=-2 u\eta \nabla \xi_j,
	$$
	$$
	\tilde{E}_j(u):=- u\xi_j\Delta \eta -2\nabla u \cdot (\nabla\eta)\xi_j + u\eta\Delta \xi_j.
	$$
	are the ``error terms''.
	The first error term will be disposed of using \eqref{e2}, while the easier second term will be dealt with using \eqref{e1}.
	
	Thus, 
	\begin{equation}
		\label{u2}
		u_j =(-\Delta)^{-1}\eta_j(-\Delta u) + \nabla(-\Delta)^{-1}E_j(u) + (-\Delta)^{-1}\tilde{E}_j(u).
	\end{equation}
	Now, we apply in both sides of \eqref{u2} $C^\infty$ mollifiers. Let $N:=[\beta]+1$, where $[\beta]$ is the largest integer at most $\beta$. We subtract from both sides the Taylor polynomial of order $N-1$ at $x=0$, and then pass to the limit. At this point we use the fact that $u_j$ is identically zero around the origin, arriving at
	\begin{align*}
		u_j & =\big[(-\Delta)^{-1}\big]_N \eta_j(-\Delta u) + \big[\nabla(-\Delta)^{-1}\big]_N E_j(u) + \big[(-\Delta)^{-1}\big]_N \tilde{E}_j(u).
	\end{align*}
	Put $V_1:=|V|+\mu \vee 1$. Multiplying the last equality by $\mathbf{1}_{B(0,1)} V_1^{\frac{1}{p}}\varphi_N$, we obtain
	\begin{align*}
		\mathbf{1}_{B(0,1)} V_1^{\frac{1}{p}}\varphi_N \eta_j |u|  \leq \; & \mathbf{1}_{B(0,1)} V_1^{\frac{1}{p}}\varphi_N \big|\big[(-\Delta)^{-1}\big]_N \mathbf{1}_{B(0,1)} \eta_j(-\Delta u)\big| \\
		+ & \mathbf{1}_{B(0,1)} V_1^{\frac{1}{p}}\varphi_N \big|\big[(-\Delta)^{-1}\big]_N \mathbf{1}_{B(0,1)}^c \eta_j(-\Delta u)\big|\\
		+& \mathbf{1}_{B(0,1)} V_1^{\frac{1}{p}}\varphi_N \big|\big[\nabla(-\Delta)^{-1}\big]_N E_j(u)\big| \\
		+ & \mathbf{1}_{B(0,1)} V_1^{\frac{1}{p}}\varphi_N \big|\big[(-\Delta)^{-1}\big]_N \tilde{E}_j(u)\big|
	\end{align*}
	or
	$$
	I \leq J+J^c+J_E+J_{\tilde{E}}.
	$$
	Note that $V_1^{\frac{1}{p}}u \in L^p$, and so $I \in L^p$, see \cite[Theorem 6.3]{LS} (the proof essentially consists of applying Proposition \ref{prop_p} to the Neumann series for $u$). In fact, 
	\begin{equation}
	\label{Vp}
	\|V_1^{\frac{1}{p}}u\|_p \leq C\|f\|_p
	\end{equation}
	(we will need this later).
	
	By \eqref{e1},
	\begin{align*}
		\|J\|_{p} & \leq C_1 \big\|\mathbf{1}_{B(0,1)} V_1^{\frac{1}{p}} (-\Delta)^{-1} |V|^{\frac{1}{p'}}\mathbf{1}_{B(0,1)} \big\|_{p \rightarrow p} \|\mathbf{1}_{B(0,1)}\varphi_N \eta_j |V|^{-\frac{1}{p'}}\Delta u\|_p \\
		& (\text{we apply Proposition \ref{prop_p}}\\
		& \text{and  use ``\eqref{delta} $\Rightarrow$ $\eta_j |\Delta u| \leq \eta_j V_1|u|$ because $\eta_j f=0$''}) \\
		& \leq C_1 \kappa_p\nu \|\mathbf{1}_{B(0,1)}\varphi_N \eta_j |V|^{\frac{1}{p}}u\|_p,
	\end{align*}
	hence $\|J\|_p \leq C_1 \kappa_p \nu \|I\|_p$.
	
	Thus, we arrive at
	\begin{equation}
		\label{main_ineq}
		(1-C_1\kappa_p\nu)\|I\|_p \leq \|J^c\|_p + \|J_E\|_p+\|J_{\tilde{E}}\|_p.
	\end{equation}
	From now on, we assume that $\nu$ is sufficiently small so that $C_1\kappa_p\nu<1$. The inequality \eqref{main_ineq} will yield the estimate \eqref{est_m} of Theorem \ref{thm1} upon taking $j \rightarrow \infty$, once we estimate the remaining terms $\|J^c\|_p$,  $\|J_E\|_p$, $\|J_{\tilde{E}}\|_p$: 
	
	1) By \eqref{e1},
	\begin{align*}
		\|J^c\|_p & \leq C_1 \big\|\mathbf{1}_{B(0,1)} V_1^{\frac{1}{p}} (-\Delta)^{-1} V_1^{\frac{1}{p'}}\mathbf{1}_{{B(0,3)}-{B(0,1)}} \big\|_{p \rightarrow p}\|\mathbf{1}_{{B(0,3)}-{B(0,1)}}\varphi_N V_1^{\frac{1}{p}}u\|_p \\
		& (\text{we have used $\supp \eta_j \subset B(0,3)$}) \\
		& \leq K_1 \|\mathbf{1}_{B(0,3)} V_1^{\frac{1}{p}}u\|_p.
	\end{align*}
	
	2) $\|J_E\|_p$ vanishes as $j \rightarrow \infty$. (This is the first place where we take into account that $u$ vanishes at least to order $\beta$ at $x=0$.) Namely, by \eqref{e2},
	\begin{align*}
		\|J_E\|_p & \leq C_2N\|\mathbf{1}_{B(0,1)} V_1^{\frac{1}{p}}(-\Delta)^{-\frac{1}{2}} \varphi_N|u||\nabla \xi_j|\|_{p} \\
		& (\text{we estimate $|\nabla \psi| \leq jc\mathbf{1}_{2j^{-1}}$ and $\varphi_N \mathbf{1}_{2j^{-1}} \leq c_1 j^N$}) \\
		& \leq C N j^{N+1}\|\mathbf{1}_{B(0,1)} V_1^{\frac{1}{p}}(-\Delta)^{-\frac{1}{2}} |u|\mathbf{1}_{2j^{-1}}\|_{p} = \dots_1
	\end{align*}
	We now apply the estimate $\mathbf{1}_{B(0,1)}(x)(-\Delta)^{-\frac{1}{2}}(x,y)\mathbf{1}_{B(0,1)}(y) \leq C_0\mathbf{1}_{B(0,1)}(x)(1-\Delta)^{-\frac{1}{2}}(x,y)\mathbf{1}_{B(0,1)}(y)$ for appropriate $C_0>1$. Strictly speaking, the latter is not necessary, but we will make the argument somewhat shorter by passing to a Bessel potential. So, we continue: for $2 \leq s<p$ if $p>2$ or $s=2$ if $p=2$, and $q=\frac{dp}{d+(1-2s^{-1})p}$,
	\begin{align}
		\dots_1 & \leq CC_0 N j^{N+1}\|\mathbf{1}_{B(0,1)} V_1^{\frac{1}{p}}(1-\Delta)^{-\frac{1}{2}} |u|\mathbf{1}_{2j^{-1}}\|_{p} \notag \\
		& \leq C C_0 N j^{N+1}\|\mathbf{1}_{B(0,1)} V_1^{\frac{1}{p}}(1-\Delta)^{-\frac{1}{s}}\|_{p \rightarrow p}\|(1-\Delta)^{-\frac{1}{2}+\frac{1}{s}}\|_{q \rightarrow p}\|u\mathbf{1}_{2j^{-1}}\|_{q} \tag{$\ast$} \label{ast}
		\\
		& (\text{we are applying $\|\mathbf{1}_{B(0,1)} V_1^{\frac{1}{p}}(1-\Delta)^{-\frac{1}{s}}\|_{p \rightarrow p} <\infty$ \cite[Theorem 6.1]{LS}}) \notag \\
		& \leq C' N j^{N+1}\|u\mathbf{1}_{2j^{-1}}\|_{q}=\dots_2 \notag
	\end{align}
	It is seen that $q$ can be chosen to be arbitrarily close $\frac{dp}{d+p-2}$ (by selecting $s$ close to $p$).
	Thus, applying H\"{o}lder's inequality, we obtain, for every $\epsilon>0$,
	\begin{align*}
		\dots_2  & \leq C' N j^{N+1}\|\mathbf{1}_{2j^{-1}}\|_{\frac{dp}{p-2}+\epsilon}\|u\mathbf{1}_{2j^{-1}}\|_p \\
		& \leq C''Nj^{N+1-\frac{p-2}{p}+\epsilon'}\|u\mathbf{1}_{2j^{-1}}\|_p \quad (\text{$\epsilon'$ is as small as needed}) \\
		& \leq C''N \bigl(j^{p(N+1)-p+2+\epsilon''}\big\langle |u|^p\mathbf{1}_{2j^{-1}}\big\rangle  \bigr)^{\frac{1}{p}}  \quad (\text{$\epsilon''$ is as small as needed}).
	\end{align*}
	Recalling that $u$ vanishes in $L^p$ at least to order $\beta$, and that $N=[\beta]+1$, we obtain that the last term in the previous formula tends to $0$ as $j \rightarrow \infty$ provided that $p([\beta]+2)-p+2<d+p\beta$, i.e.\,$p([\beta]-\beta)+p+2<d$. (Recall $d=3$.) Since the last condition is satisfied by the assumptions of the theorem, we can make $\|K_E\|_p$ as small as needed by selecting $j$ sufficiently large. (Let us note that if $p([\beta]+2)-p+2>d+p\beta$, then we cannot exclude the possibility that $\|K_E\|_p \rightarrow \infty$ as $j \rightarrow \infty$.)
	
	\smallskip
	
	Next, recall that $\tilde{E}_j(u):=- u\xi_j\Delta \eta -2\nabla u \cdot (\nabla\eta)\xi_j + u\eta\Delta \xi_j$.
	
	3) By \eqref{e1},
	\begin{align*}
		\|J_{\tilde{E}}\|_p \leq A_1+A_2+A_3,
	\end{align*}
	where
	\begin{align*}
		A_1& :=C_1\|\mathbf{1}_{B(0,1)} V_1^{\frac{1}{p}}(-\Delta)^{-1}\varphi_N |u|\xi_j|\Delta \eta|\|_p \\
		& \leq C_1C_0c^2  \|\mathbf{1}_{B(0,1)} V_1^{\frac{1}{p}}(1-\Delta)^{-1}|u|\mathbf{1}_{B(0,3)}\|_p \\
		& (\text{we represent $(1-\Delta)^{-1}=(1-\Delta)^{-\frac{1}{s}} (1-\Delta)^{-1+\frac{1}{s}}$ and argue as in \eqref{ast}}) \\
		& \leq K_2\|\mathbf{1}_{B(0,3)}u\|_p,
	\end{align*}
	and
	\begin{align*}
		A_2 & :=2C_1\|\mathbf{1}_{B(0,1)} V_1^{\frac{1}{p}}(-\Delta)^{-1}\varphi_N |\nabla u| |\nabla\eta|\xi_j\|_p \\
		& \leq 2C_1C_0c \|\mathbf{1}_{B(0,1)} V_1^{\frac{1}{p}}(1-\Delta)^{-1}|\nabla u| \mathbf{1}_{B(0,3)}\|_p \\
		& \leq K_3\|\mathbf{1}_{B(0,3)}\nabla u \|_r, \quad r=\frac{dp}{d+2p-2}+\epsilon, 
	\end{align*}
	where $\epsilon>0$ can be chosen arbitrarily small. (It is easy to see $\|\mathbf{1}_{B(0,3)}\nabla u \|_r$ is finite: since $d=3$, $r=\frac{3p}{2p+1}+\epsilon<2$, so H\"{o}lder's inequality and $u \in W^{1,2}$ immediately yield the conclusion.)

	Finally, we take care of $A_3$. (This is the second place where we take into account that $u$ vanishes at least to order $\beta$ at $x=0$.)
	\begin{align*}
		A_3:=&C_1\|\mathbf{1}_{B(0,1)} V_1^{\frac{1}{p}}(-\Delta)^{-1}\varphi_N |u|\eta|\Delta \xi_j|\|_p \\
		& \leq C_1C_0 c^2 j^{N+2}\|\mathbf{1}_{B(0,1)} V_1^{\frac{1}{p}}(1-\Delta)^{-1}|u|\mathbf{1}_{2j^{-1}}\|_p \\
		& (\text{again, we write $(-\Delta)^{-1}=(-\Delta)^{-\frac{1}{s}} (-\Delta)^{-1+\frac{1}{s}}$ and argue as in \eqref{ast}}) \\
		& \leq C'j^{N+2-\frac{2p-2}{p}+\epsilon}\|u\mathbf{1}_{2j^{-1}}\|_p \quad (\text{$\epsilon$ is as small as needed}).
	\end{align*}
	Since $u$ vanishes in $L^p$ at least to order $\beta$ at $x=0$, we have $j^{N+2-\frac{2p-2}{p}+\epsilon}\|u\mathbf{1}_{2j^{-1}}\|_p \rightarrow 0$ as $j\rightarrow \infty$ provided that $p([\beta]+3)-2p+2<d+p\beta$. That is, $p([\beta]-\beta)+p+2<d$, i.e.\,we arrive at the same condition as above. Hence $A_3$ can be made as small as needed by selecting $j$ sufficiently large.

	Applying the above estimates to \eqref{main_ineq}, we thus obtain
	\begin{align*}
		(1-C_1\kappa_p\nu)\|I\|_p & \leq K_1\|\mathbf{1}_{B(0,3)} V_1^{\frac{1}{p}}u\|_p \\
		& + K_2\|\mathbf{1}_{B(0,3)}u\|_p + K_3\|\mathbf{1}_{B(0,3)}\nabla u \|_r \\
		&+ \text{ terms that vanish as $j \rightarrow \infty$}.
	\end{align*}
	Taking $j \rightarrow \infty$, we have
	\begin{align*}
		\|\mathbf{1}_{B(0,1)} V_1^{\frac{1}{p}}\varphi_N  u \|_p \leq \tilde{K}, 
	\end{align*}
	where $\tilde{K}:=K_1\|\mathbf{1}_{B(0,3)} V_1^{\frac{1}{p}}u\|_p  + K_2\|\mathbf{1}_{B(0,3)}u\|_p + K_3\|\mathbf{1}_{B(0,3)}\nabla u \|_r$.
	Hence $\|\mathbf{1}_{B(0,1)} \varphi_N  u \|_p \leq \tilde{K}$ (recall $N=[\beta]+1$).

It remains  to show that $\tilde{K} \leq K$ for some $K=K(\|f\|_p,\|f\|_2,\nu)$. Indeed, in view of \eqref{Vp}, the first two terms in $\tilde{K}$ are bounded from above by $K_4\|f\|_p$. Next, since $r<2$, $\|\mathbf{1}_{B(0,3)}\nabla u \|_r \leq \|\nabla u\|_2 \leq K_5\|f\|_2$, where the last inequality is valid by the construction of the form-sum.
		\hfill \qed

	\bigskip

	\begin{remark}
	\label{proof_d}
	Here we comment on the proof of the result in Remark \ref{rem_d} ($d \geq 4$).
	Following the proof of the SUC in \cite{KSh}, we arrive at the error term  $\|\mathbf{1}_{2/j\setminus 1/j}\varphi_{N_\delta}(-2\nabla\xi_j\cdot\nabla u) \|_q$, $q=\frac{2d}{d+2}$, which is estimated using the Gagliardo-Nirenberg inequality
	$$
		\|\mathbf{1}_{B(0,2/j)- B(0,1/j)}\varphi_{N_\delta}(-2\nabla\xi_j\cdot\nabla u) \|_p\leq j^{N_\delta+1}\|\mathbf{1}_{2/j}\nabla u \|_p\leq C j^{N_\delta+1}\|\mathbf{1}_{2/j}\Delta u \|_{\frac{2d}{d+4}} + Cj^{N_\delta+1}\|\mathbf{1}_{2/j} u \|_2,
	$$
	where $N_\delta:=[\beta]+1+\left(\frac{d}{2}-\gamma\right)\frac{d-3}{d-1}$ for any $0<\gamma<\frac{1}{2}$.
	The latter term vanishes as $j\rightarrow\infty$ since $2([\beta]-\beta)+4+2\left(\frac{d}{2}-\gamma\right)\frac{d-3}{d-1}<d$ (indeed, by our assumption $[\beta]-\beta+2+\bigl(\frac{d}{2}-\gamma\bigr)\frac{d-3}{d-1} <\frac{d}{s}$ for $\gamma$ sufficiently close to $\frac{1}{2}$). For the former term, we estimate
	$$
	\|\mathbf{1}_{2/j}\Delta u \|_{\frac{2d}{d+4}} \leq \|\mathbf{1}_{2/j}V_1 u \|_{\frac{2d}{d+4}}\leq \|\mathbf{1}_{2/j}V_1\|_{\frac{d-1}{2}}\|\mathbf{1}_{2/j} u\|_s, \quad s=2\frac{d(d-1)}{d^2-d-4},
	$$
	where $j^{N_\delta+1}\|\mathbf{1}_{2/j} u\|_s \rightarrow 0$ as $j \rightarrow \infty$ by our assumption on the order of vanishing of $u$, i.e.\,$[\beta]-\beta+2+\bigl(\frac{d}{2}-\gamma\bigr)\frac{d-3}{d-1} <\frac{d}{s}$.

	\end{remark}

	\section{Further discussion}
	
	\label{discussion_sect}
	
	1.~Repeating the proof of Theorem \ref{thm1} for solutions to the differential inequality \eqref{diff_ineq} below, we obtain the following strong unique continuation result, strengthening the corresponding result in \cite{KSh} (i.e.\,removing the extra assumption $V \in L^{1+\varepsilon}_{\loc}$ for some $\varepsilon>0$).

	\begin{theorem}\label{cor1} Let $d=3$. Assume that $V \in L^1_{\loc}$ satisfies 
			\begin{equation}
			\label{unif_v}
			\sup_{x \in \mathbb R^d}\|\mathbf{1}_{B(x,\rho)}|V|^{\frac{1}{2}}(-\Delta)^{-\frac{1}{2}}\|_{2 \rightarrow 2} \leq \nu
		\end{equation}
		for a sufficiently small $\nu$.
		Then any solution 
		$$u \in Y^{\rm str}_V:=\{f \in L^1_{\loc} \mid \Delta f \in L^1_{\loc}, |V|^{\frac{1}{2}}f \in L^2_{\loc}, |\nabla f| \in L^{\frac{6}{5}}\}$$ to the differential inequality
		\begin{equation}
			\label{diff_ineq}
			\tag{$\star\star$}
			|\Delta u| \leq |Vu| \quad \text{a.e.\,in } \mathbb R^3
		\end{equation}
		that vanishes in $L^2$ to infinite order at some point is, in fact, identically equal to zero on $\mathbb R^3$.
	\end{theorem}
	
	The space of solutions $Y^{\rm str}_V$ is large enough to contain the eigenfunctions of the Schr\"{o}dinger operator $H$, see \cite{KSh} for details.
	
	\medskip
	
	2.~As a consequence of \eqref{est_m} in Theorem \ref{thm1}, we have for any $0<a<1$
				$$\ord_{x=0}^pu \leq \log_{1/a}\frac{K}{\|\mathbf{1}_{B(0,a)}u\|_p} + 1.$$
			 Indeed, we select $\beta$ such that $\beta \leq \ord_{x=0}^pu \leq \beta+1$ and note that by \eqref{est_m}
	$$
	\|\mathbf{1}_{B(0,a)}a^{-[\beta]-1}u\|_p \leq K,
	$$
	and so
	$$
	(1/a)^{[\beta]+1}\|\mathbf{1}_{B(0,a)}u\|_p \leq K \quad \Rightarrow \quad [\beta]+1 \leq \log_{1/a} \frac{K}{\|\mathbf{1}_{B(0,a)}u\|_p},
	$$
	as required.

\medskip
	
	3.~In \cite{MSS, MNS} the authors consider operators $\Lambda=-\Delta - b \cdot \nabla$, $\Lambda^*=-\Delta + \nabla \cdot b$, where $b(x)=\sqrt{\delta}\frac{d-2}{2}|x|^{-2}x$ is the Hardy-type drift, and establish two-sided bounds on their Green functions:
	$$
	(\mu+\Lambda)^{-1}(x,y) \simeq e^{-\sqrt{\mu}|x-y|}|x-y|^{-d+2}\biggl[1 \wedge \frac{|x||y|}{|x-y|^2} \biggr]^{\frac{\gamma}{2}}|x|^{-\frac{\gamma}{2}}|y|^\frac{\gamma}{2}, \qquad \gamma:=\sqrt{\delta}\frac{d-2}{2},
	$$
	$$
	(\mu+\Lambda^*)^{-1}(x,y) \simeq e^{-\sqrt{\mu}|x-y|}|x-y|^{-d+2}\biggl[1 \wedge \frac{|x||y|}{|x-y|^2} \biggr]^{\frac{\gamma}{2}}|x|^{\frac{\gamma}{2}}|y|^{-\frac{\gamma}{2}}.
	$$		
Since we are interested in the vanishing of the Green function in $x$, is the operator $\Lambda^*$ that is of interest to us. The proof of estimate \eqref{est_m} in Theorem \ref{thm1} extends to this operator easily (use \eqref{e2} instead of \eqref{e1}), although with the assumption $\nu$ dependent on the order of vanishing.

The phenomenon of vanishing of Green's function exists in the non-local setting, see  \cite{BGJP, CKSV, JW} (fractional Laplacian $(-\Delta)^{\frac{\alpha}{2}}$ with Hardy potential $c|x|^{-\alpha}$) and \cite{KS, KSS} (Hardy-type drift  $c \nabla \cdot |x|^{-\alpha}x$), and so the question addressed in this paper is also of interest for the operators considered in these papers. 

	\appendix
	
	\bigskip

	\section{Proof of \eqref{bd_0}}
	
	\label{appA}

	Set $|x|_\varepsilon:=\sqrt{|x|^2+\varepsilon}$ for $\varepsilon>0$ and
	$$
	V_\varepsilon(x):=\delta\frac{(d-2)^2}{4}|x|_\varepsilon^{-2} + V_{0,\varepsilon},
	$$
	where $V_{0,\varepsilon} \in C_c^\infty$ is such that $|V_{0,\varepsilon}| \uparrow |V|$ as $\varepsilon \downarrow 0$, so that $V_{0,\varepsilon} \in \mathbf{F}_{\delta_0}$ with the same constants $\delta_0$, $\lambda_0=\lambda(\delta_0)$. Such $V_{0,\varepsilon}$ can easily be constructed by truncating and mollifying $V_0$.
	Let
	$$
	H^\varepsilon:=-\Delta + V_\varepsilon, \quad D(H^\varepsilon)=W^{2,2},
	$$
	$$
	u \equiv u_\varepsilon :=(\mu + H^\varepsilon)^{-1}f, \quad 0 \leq f \in L^1 \cap L^\infty, f \not \equiv 0, f=0 \text{ in $B(0,1)$},
	$$
	$$
	\psi(x) \equiv \psi_\varepsilon(x):=|x|_{\varepsilon}^{-s}, \quad 0 \leq s<\frac{d-2}{2}\sqrt{1+\delta-\delta_0}, \qquad v:=\psi u.
	$$
	Then 
	$v$ satisfies
	$$
	(\mu + H^\varepsilon_\psi)v=\psi f, \quad \text{ where } H^\varepsilon_\psi:=-\Delta + 2 \frac{\nabla\psi}{\psi} \cdot \nabla + \frac{\Delta \psi}{\psi} - 2\frac{(\nabla \psi)^2}{\psi^2} + V_\varepsilon, \quad D(H^\varepsilon_\psi)=W^{2,2}.
	$$
	We will show below that $H^\varepsilon_\psi$
	satisfies 
	\begin{equation}
		\label{S}
		\|\nabla v\|_{2} \leq c\|\psi f\|_2, \quad \mu>\delta_0\lambda_0.
	\end{equation}
	This inequality yields $\|\psi u\|_{\frac{2d}{d-2}} \leq C_Sc\|\psi f\|_2<\infty$, so
	$$
	\|\mathbf{1}_{B(0,1)}|x|_\varepsilon^{-\frac{k}{p}}u_\varepsilon\|_{p}<\infty, \quad \text{ where } p=\frac{2d}{d-2}, \frac{k}{p}=s.
	$$
	Now, using the standard convergence result $(\mu+H)^{-1}=s{\mbox-}L^2{\mbox-}\lim_{\varepsilon \downarrow 0}(\mu+H^\varepsilon)^{-1}$ and applying Fatou's Lemma, we obtain $\|\mathbf{1}_{B(0,1)}|x|^{-\frac{k}{p}}u\|_{p}<\infty$.
	Comparing the latter with our hypothesis on $s$, we obtain that $\Ord_{x=0}^{p}u \geq \frac{d-2}{2}(\sqrt{1+\delta-\delta_0}-1)$, as claimed.
	
	The inequality \eqref{S} will follow from $\mu \langle w,w\rangle + \Real\langle H^\varepsilon_\psi w,w\rangle \geq (\mu-\delta_0\lambda_0)\|w\|_2^2 + c\|\nabla w\|^2_{2}$, $c>0$, applied to $w=\Real(v)$.
	By integration by parts,
	$$
	2\Real\langle \frac{\nabla\psi}{\psi} \cdot \nabla w,w \rangle \equiv -2s\Real\langle |x|_\varepsilon^{-2}x\cdot \nabla w,w\rangle = s(d-2)\langle |x|_\varepsilon^{-2}w,w\rangle + 2s\varepsilon \langle |x|_\varepsilon^{-4}w,w\rangle.
	$$
	Thus, using Hardy's inequality $\langle \nabla w,\nabla w\rangle \geq \frac{(d-2)^2}{4}\langle |x|^{-2}w,w\rangle \geq \frac{(d-2)^2}{4}\langle |x|_\varepsilon^{-2}w,w\rangle$, we obtain
	\begin{align*}
		\Real\langle H^\varepsilon_\psi w,w\rangle & \geq  \langle \nabla w,\nabla w\rangle + \left [ -s^2 + \delta\frac{(d-2)^2}{4}\right] \langle |x|_\varepsilon^{-2}w,w\rangle + s^2\varepsilon \langle |x|_\varepsilon^{-4}w,w\rangle  + \langle V_0w,w\rangle \\
		& (\text{we discard the term $s^2\varepsilon \langle |x|_\varepsilon^{-4}w,w\rangle$ and use $V_0 \in \mathbf{F}_{\delta_0}$}) \\
		& \geq (1-\delta_0)\langle \nabla w,\nabla w\rangle + \left [ -s^2 + \delta\frac{(d-2)^2}{4}\right] \langle |x|_\varepsilon^{-2}w,w\rangle - \delta_0\lambda_0\langle w,w\rangle \\
		& \geq c\langle \nabla w,\nabla w\rangle + (1-\delta_0-c)\frac{(d-2)^2}{4}\langle |x|_\varepsilon^{-2}w,w\rangle \\
		& + \left [ -s^2 + \delta\frac{(d-2)^2}{4}\right] \langle |x|_\varepsilon^{-2}w,w\rangle - \delta_0\lambda_0\langle w,w\rangle
	\end{align*}
	where $c$ is chosen sufficiently small so that $1+\delta-\delta_0-c>s^2\frac{4}{(d-2)^2}$, using the fact that our assumption on $s$ is a strict inequality. Then $\Real\langle H^\varepsilon_\psi w,w\rangle \geq c\langle \nabla w,\nabla w\rangle - \delta_0\lambda_0\langle w,w\rangle$, as needed.

\end{document}